\documentclass[preprint,11pt]{elsarticle}

\usepackage{graphicx}
\usepackage{amsmath,amssymb}

\makeatletter
\newcommand{\tpitchfork}{%
  \vbox{
    \baselineskip\z@skip
    \lineskip-.52ex
    \lineskiplimit\maxdimen
    \m@th
    \ialign{##\crcr\hidewidth\smash{$-$}\hidewidth\crcr$\pitchfork$\crcr}
  }%
}
\makeatother

 \usepackage{amsthm}
\usepackage[all]{xy}
\newcommand{\mathsout}[1]
{\bgroup\mathchoice
  {\sbox0{$\displaystyle{#1}$}%
    \usebox0\hspace{-\wd0}%
    \rule[0.5\ht0-0.5\dp0-.5pt]{\wd0}{1pt}}%
  {\sbox0{$\textstyle{#1}$}%
    \usebox0\hspace{-\wd0}%
    \rule[0.5\ht0-0.5\dp0-.5pt]{\wd0}{1pt}}%
  {\sbox0{$\scriptstyle{#1}$}%
    \usebox0\hspace{-\wd0}%
    \rule[0.5\ht0-0.5\dp0-.5pt]{\wd0}{1pt}}%
  {\sbox0{$\scriptscriptstyle{#1}$}%
    \usebox0\hspace{-\wd0}%
    \rule[0.5\ht0-0.5\dp0-.5pt]{\wd0}{1pt}}%
\egroup}
\usepackage{lineno}
\usepackage{enumerate}

\newtheorem{thm}{Theorem}[section]

\newtheorem{cor}[thm]{Corollary}
\newtheorem{lem}[thm]{Lemma}
\newtheorem{prop}[thm]{Proposition}
\newtheorem{defn}[thm]{Definition}
\newtheorem{rem}{Remark}[section]

\newtheorem{asmp}{Assumption}
\newcommand{\eps}{\varepsilon}

\newcommand{\chro}{\stackrel{\longrightarrow}{{\rm exp}}\int}

\newcommand{\pr}{\mbox{\rm pr}}
\newcommand{\ad}{\mbox{ad}}

\journal{Journal Name}

\begin{document}

\begin{frontmatter}


\title{Control  on the Manifolds of Mappings with a View to the Deep Learning}


\author[label1]{Andrei Agrachev}
 \address[label1]{SISSA, via Bonomea 265, Trieste, 34136, Italy; agrachev@sissa.it}
\author[label2]{Andrey Sarychev}
 \address[label2]{DiMaI, via delle Pandette 9, Firenze, 50127,  Italy; andrey.sarychev@unifi.it}

\begin{abstract}
Deep learning of the Artificial Neural Networks (ANN) can be treated as a particular
class of interpolation problems. The goal is to find a neural network
whose input-output map approximates well the desired map
on a finite or an infinite training set. Our idea consists of taking
as an approximant the input-output map, which arises from a nonlinear
continuous-time control system. In the limit such control system can be
seen as a network with a continuum of layers, each one labelled by the
time variable. The values of the controls at each instant of time are the
parameters of the layer.
\end{abstract}


\end{frontmatter}



\section{Introduction and problem setting}
The name {\it deep learning}  stands for a set of the methods and the tools
which study the  problems of classification such as image recognition, speech recognition etc.
These   methods  involve
multilayered artificial neural networks (ANN) and one of the key moments
is the {\it training} of the networks on a set of classified objects.
For a simple mathematical model  of the   multilayered ANN and of the process
of  its  training we refer to  \cite{HiH}.

The functioning of the ANN results from
a composition of
the actions of separate neurons.
Each neuron realizes an activation    function $\sigma : \mathbb{R} \rightarrow \mathbb{R}$ with parameters.
There are plenty of  choices for
 the activation   function, that is normally nonlinear monotone  sigmoid-like function. 
The vector functions can be assembled from the scalar  activation functions:
\begin{equation}\label{eq_bar_sigma}
  \bar \sigma :  \mathbb{R}^m \rightarrow  \mathbb{R}^m :    \  \bar \sigma (x_1, \ldots , x_m)=(\sigma(x_1), \ldots , \sigma(x_m)) .  \end{equation}
One can assemble neurons in a multi-layer network  in such a way that the outputs of the neurons from a previous layer
serve as the inputs for the successive level.

One can   introduce   parameters into the activation   functions  via a  substitution of their variables.
For example a linear change of the argument in \eqref{eq_bar_sigma}
results in    $\bar \sigma(Kx+B)$, where $x \in \mathbb{R}^m, \ K \in \mathbb{R}^{m \times m}, \ B  \in \mathbb{R}^m$.

The output of an ANN realizes the  composition of the functions, each one of which corresponds  to a layer:
 \begin{eqnarray}\label{eq_iter}
 F(x)&=& \\
 &=&\bar \sigma \left(K^{[M]}\bar \sigma\left(K^{[M-1]}\left(    \ldots \bar\sigma(K^{[1]}x+B^{[1]}) \ldots \right)+B^{[M-1]}\right) +B^{[M]}\right). \nonumber
\end{eqnarray}

To  set the classification problem    we consider   a finite set of  objects, which are  represented   by the
vectors $x^i \in \mathbb{R}^d, \ i \in \mathcal{I}$. Let $X=\{x^i| i \in \mathcal{I}\}$.
There is    a $\mathbb{R}^s$-valued  classifying function
$c: \mathbb{R}^d \mapsto \mathbb{R}^s$,  defined
on  $X$, which attributes to each object $x^i$ its "class"  $c(x^i) \in \mathbb{R}^s$.

The objective  of the training of an ANN amounts to the adjustment of
the values of the parameters    $K^{[1]}, \ldots , K^{[M]}$,  $B^{[1]}, \ldots , B^{[M]}$ in order to achieve
the best approximation of the classifying function $c(x)$  by the output map \eqref{eq_iter}.
More specifically
one    seeks to   minimize   the  value of the  {\it loss  function}, which measures the discrepancy between the input-output map of the system and the classification function. For example the least square  loss function has form
    \begin{equation}\label{eq_cost}
    C\left(K^{[M]}, \ldots , K^{[1]}, B^{[M]}, \ldots , B^{[1]}\right)=\sum_{i=1}^N \left\|c(x^i)-F(x^i)   \right\|^2_2  \rightarrow \min_{K^{[j]}, B^{[j]}}. 
\end{equation}

  Minimization of    \eqref{eq_cost} results in  a problem of nonlinear programming,
  which   even for a "medium" number  of layers
can turn   rather complex for    classical   approaches. 

In this contribution we
 base on a continuous-time    dynamic or {\it residual network} model for deep learning,
with a continuum of layers,  labelled  by the time variable.
The parameters involved at each layer are the values of the controls at the respective instant of time.
The analogue of the composition \eqref{eq_iter}
is  the end-point of the trajectory or the output of the continuous-time control system in their dependence  on  control.

As in the model, we referred to above, in the control-theoretic setting  one seeks for the values of the parameters (the controls), which
provide the  best approximation of the classifying function by  the output of the control system.
Precise formulations and the description of the  model can be found  in section \ref{sec_ocm}.

The  setting   allows
 for the  application of  analytic methods of {\it dynamic optimization}
such as dynamic programming, Bellman's optimality principle and Pontryagin's maximum principle
together with the corresponding  numerical algorithms.
This approach to the deep learning has been initiated in the last years by a number of scholars;
see for example \cite{EGPZ,LCTE,TG} and references therein. The readers must be warned that  we consider a very
restricted issue of possible application of the methods of ensemble controllability and ensemble optimal control   to the problems of deep learning. Therefore we only cite the references related to this concrete topic, leaving aside not only a huge amount of literature on deep learning, but also on application of the methods
of deep learning to the problems of optimal control.

In the contribution we concentrate on finding the  classes  of control systems, which are able
to guarantee approximation of the classifying functions  at each rate.
It  amounts to   studying the problems of ensemble controllability of the control systems and
 the action of the  flows, generated by the control systems,   on  the manifold of mappings.
We formulate  the sufficient  criteria  (Theorem \ref{thm_grodif}, Corollary \ref{thm_maps})  of ensemble controllability and
provide  examples of the nonlinear control systems which demonstrate approximate controllability property
in the group of diffeomorphisms  of $\mathbb{R}^n$ (Theorem~\ref{thm_GH}), of a torus $\mathbb{T}^n$ (Theorem~\ref{thm_Td} ) and of the $2$-dimensional sphere $\mathbb{S}$ (Theo\-rems \ref{thm_SDiff} and \ref{thm_Diff}).

\section{Neural networks modelled by  control systems}
It is  an easy task   to reformulate optimization  problem \eqref{eq_iter}-\eqref{eq_cost}  as an
optimal control problem for a discrete-time  controlled  dynamic  system.
If one  sets  the variables $z_1, z_2, \ldots , z_M$, which satisfy  the relations
 \begin{equation}\label{eq_dtcs}
   z_1=x, \ z_{j+1}=\bar \sigma \left(K^{[j+1]}z_j+B^{[j+1]}\right), \ j=1, \ldots , M-1,
 \end{equation}
 then the  map, defined by  \eqref{eq_iter}, coincides with the "end-point map"
 \begin{equation}\label{eq_fxzn}
   F(x,K^{[1]}, \ldots K^{[M]},  B^{[1]}, \ldots ,B^{[M]})=z_M.
 \end{equation}

Alternatively one can introduce the intermediate variables $y_j$ and define the dynamics
 \begin{equation}\label{eq_2dtcs}
  z_1=x, \  y_{j+1}= K^{[j+1]}z_j+B^{[j+1]}, \    z_{j+1}=\bar \sigma \left( y_{j+1}\right), \ j=1, \ldots , M-1,
 \end{equation}
 getting again  formula \eqref{eq_fxzn} for the map \eqref{eq_iter}.


Denote  $z_j^i$  (respectively $y_j^i,z_j^i$)   the points of the trajectories  of
equation \eqref{eq_dtcs}  (respectively \eqref{eq_2dtcs}),  which start  with the initial data
$z^i_1=x^i \in X, \ i \in \mathcal{I}$. Then    the problem of the best least square approximation  \eqref{eq_cost}  takes the form
  \begin{equation}\label{eq_zcost}
    \hat{C}\left(K^{[2]}, \ldots , K^{[M]}, B^{[2]}, \ldots , B^{[M]}\right)=\sum_{i\in \mathcal{I}} \left\|c(x^i)-z^i_M   \right\|^2_2  \rightarrow \min .
  \end{equation}

Problems \eqref{eq_dtcs},\eqref{eq_zcost},
respectively \eqref{eq_2dtcs},\eqref{eq_zcost}  are Mayer problems of optimal control for the control systems with discrete time with free end-point.
There are quite few  numerical  algorithms developed for this class of problems,
but we do not treat them in this contribution, making emphasis  instead  on the continuous-time control systems.
\footnote{It is worth mentioning    that  the theoretical study of the
discreet-time optimal control problems manifests additional   complexities
in comparison with the continuous-time case, unless additional regularity assumptions, such as convexity are imposed.}

The way  to the representation of the input-output map
\eqref{eq_2dtcs}-\eqref{eq_fxzn}
as an output of a continuous-time control system is rather straightforward.
Let us consider a  system, which for the sake of the computational simplicity we choose control-linear:
\begin{equation}\label{eq_ctcs}
 \dot{z}=f^0(z)u_0(t)+\sum_{i=1}^{r}f^i(z)u_i(t), z \in \mathbb{R}^m.
\end{equation}
For the purpose of our illustration  we choose smooth vector field $f^0(z)$ to be nonlinear, and the vector fields $f^1(z), \ldots , f^r(z)$,  to  form a basis of the space of the affine vector fields in $\mathbb{R}^m$.

Require  the diffeomorphism
$e^{f^0(z)}$ to coincide with  $\bar \sigma (z)$, so that the map $\bar \sigma (z)$ is generated by  control system \eqref{eq_ctcs}, driven by the constant control
$u(t)=(1,0, \ldots , 0)$ on a unit time interval. Each affine  diffeomorphism $Kx+B$, with $\det K >0$,    can be represented as a composition of the diffeomorphisms $e^{a(z)}$, where $a(z)$, are  affine vector fields in $\mathbb{R}^m$. Hence such diffeomorphisms are generated by the control system \eqref{eq_ctcs}, driven by the piecewise-constant controls.

 Therefore  the composition \eqref{eq_iter} or, the same the
output map \eqref{eq_fxzn} of the discrete-time system \eqref{eq_2dtcs}
can be represented as the endpoint map of the continuous-time system \eqref{eq_ctcs}, driven by a piecewise-constant control.

\section{Ensemble Optimal Control Model for the training of control-theoretic ANN}
\label{sec_ocm}

\subsection{Ensemble Optimal Control Model}
We consider a  training set $X=\{x^1, \ldots , x^N\} \subset \mathcal{M}$, consisting of $N$ points of
a connected Riemannian manifold $\mathcal{M}$.
In what follows
     $\mathcal{M}$ will be  a  submanifold of $R^d$.

      We  set  an optimal control  model for the  training process of an ANN, which   involves a  control system in $\mathbb{R}^d$, which 
\begin{equation}\label{eq_cafs}
  \dot{x}=\sum_{i=1}^{r}f^i(y)u_i(t), \ y \in \mathcal{M}.
\end{equation}

We introduce the  terminology of the ensembles of points.
A finite ensemble of points of a smooth manifold $\mathcal{M}$  is  an $N$-ple  $\gamma =(x^1, \ldots , x^N) \in \mathcal{M}^N$,
 whose  components $x^j \in \mathcal{M}$ are pairwise distinct: $ i \neq  j  \Rightarrow   x^i \neq x^j   $.
Thus if  $\Delta^N \subset \mathcal{M}^N$  stands for  the set of $N$-ples $(x^1, \ldots , x^N) \in \mathcal{M}^N$ with (at least) two coinciding components,
then the space  $\mathcal{E}_N(\mathcal{M})$ of the ensembles of $N$ points of $\mathcal{M}$ is   the complement of $\Delta^N: \ \mathcal{E}_N(\mathcal{M})=\mathcal{M}^N \setminus \Delta^N=\mathcal{M}^{(N)}$.
Note that whenever  $\dim \mathcal{M} >1, \  \mathcal{M}^{(N)}$ is
an open connected subset and a submanifold of $\mathcal{M}^N$.

     Introduce  a classifying map  $c: X \to  \mathcal{C}$, where $\mathcal{C}$ is a connected Riemannian manifold.

Our goal is to approximate  the  map $c$ by an action of the
flow $P_t$,  generated by  the control system \eqref{eq_cafs}
which is driven by  a control $u(t)=(u_1(t), \ldots , u_r(t))$.
The flow $P_t$ acts on  an ensemble  $(x^1, \ldots , x^N)$ as $$P_t(x^1, \ldots , x^N)=(z_1(t), \ldots , z_k(t))$$ where
 $z_k(t)$ are the points
of the trajectories of the Cauchy problems
\begin{eqnarray}\label{eq_N_ens}
  \dot{z}_k=\sum_{i=1}^{r}f^i(z_k)u_i(t), \ k=1, \ldots , N, \\
\label{eq_inic_gen}
  z_k(0)=x^k, \ k=1, \ldots , N .
\end{eqnarray}
We introduce an output map
 \[ p: \mathcal{M} \to \mathcal{C}  ,\]
 which is a submersion in the cases, we consider.

We fix $T>0$ and   seek to minimize
\begin{equation}\label{eq_discrep0}
  \frac{1}{2}\sum_{k=1}^N\left\|p(z_k(T))-c(x^k)\right\|^2
\end{equation}
under constraints \eqref{eq_N_ens}-\eqref{eq_inic_gen}.

The infimum of \eqref{eq_discrep0} is either positive or null.
The distinction is related
to the  presence or the lack of controllability
of  system \eqref{eq_cafs} in the space of  finite  ensembles of points.
The problems of controllability have been addressed
in  \cite{AS20},
where we
 arranged examples of the systems, which are controllable in the space of  finite  ensembles of points. We  proved
 that   for   arbitrary $N$     generic $r$-ples of vector fields
$f^1(z), \ldots , f^r(z) \in \mbox{Vect}\left(R^d\right)$ manifest  this property.

Note that even for ensemble controllable systems,
the greater is $N$,
more complex are the controls
$u_1(t), \ldots , u_r(t)$, which are needed to achieve controllability.

For this reason  we opt  for a tradeoff between the rate or quality of the approximation (minimization of \eqref{eq_discrep0})
 and the complexity of the needed control, introducing   the loss functional $\mathcal{J}$
\begin{equation}\label{eq_functional}
 \mathcal{J}=  \frac{1}{2}\sum_{k=1}^N\left\|p(z_k(T))-c(x^k)\right\|^2 +\frac{\beta}{2}\int_0^T\left(\sum_{i=1}^{r}|u_i(t)|^2\right)dt  \rightarrow \min  .
\end{equation}

 Problem  \eqref{eq_N_ens}-\eqref{eq_inic_gen}-\eqref{eq_functional} is   Bolza optimal control problem with free end-point.
   In what regards study of the optimal control problem   we limit ourselves
 to the formulation  (in the following subsection)
of the first-order optimality condition for the problem.
In the rest of the contribution
we concentrate on the problems of ensemble controllability.

\subsection{Equations of Pontryagin Maximum Principle for Ensemble Optimal Control Problem \eqref{eq_N_ens}-\eqref{eq_inic_gen}-\eqref{eq_functional}}
We start  introducing  the pre-Hamiltonian for \eqref{eq_N_ens},\eqref{eq_functional}
\begin{equation}\label{eq_preHam}
  H=\sum_{i=1}^r \left(\sum_{k=1}^N \psi_k f_i(z_k)\right)u_i- \frac{\beta}{2}\left(\sum_{i=1}^r u_i^2\right) ,
\end{equation}
where $\psi_k \in R^{d*}, \ k=1, \ldots , N$.

The adjoint equations of the corresponding pre-Hamiltonian system are
\begin{equation}\label{eq_adjoint}
  \dot{\psi}_k=-\frac{\partial H}{\partial z_k}=-\psi_k\sum_{i=1}^r \frac{\partial f_i}{\partial z}(z_k)u_i(t), \ k=1, \ldots , N.
\end{equation}
The  end-point conditions for the adjoint variables are
\begin{equation}\label{eq_epsi}
  \psi_k(T)=-(p(z_k(T))-c(x^k))^*\frac{\partial p}{\partial z}(z_k(T)), \ k=1, \ldots , N.
\end{equation}
Let  $z=(z_1, \ldots , z_N)$,  $\psi=(\psi_1, \ldots , \psi_N)$, $u=(u_1, \ldots , u_r)$.
 Introducing the functions
  \[F_i(x,\psi)=\sum_{k=1}^N  \psi_k  f_i(z_k), \ i=1, \ldots , r\]
we bring  the pre-Hamiltonian  \eqref{eq_preHam}  to the form
\begin{equation}\label{eq_preHam_FG}
H(z,\psi,u)=\sum_{i=1}^r F_i(z,\psi)u_i- \frac{\beta}{2}\left(\sum_{i=1}^r u_i^2\right). 
\end{equation}
According to the Pontryagin's Maximum Principle if $\tilde u(t), \tilde z(t)$ are the  optimal control and the corresponding optimal trajectory
 of the problem, then there must exist $\beta \geq 0$ and an   adjoint covector $\tilde \psi(t)$, which satisfy the equations \eqref{eq_adjoint} and \eqref{eq_epsi} and such that
 \[H(\tilde z(t),\tilde \psi(t), \tilde u(t))=\max_u H(\tilde z(t),\tilde \psi(t), u). \]
 By the maximality condition we get $\frac{\partial H}{\partial u_i}|_{(\tilde u(t), \tilde z(t))}=0, \ i=1, \ldots r$, which  in the normal ($\beta >0$) case implies:
\begin{equation}\label{eq_uiFG}
  u_i=\beta^{-1}F_i(z,\psi), \ i=1, \ldots ,r.
\end{equation}


Substituting expressions \eqref{eq_uiFG}  into pre-Hamiltonian \eqref{eq_preHam_FG}  we obtain the maximized (with respect to $u$) Hamiltonian
\begin{equation*}\label{eq_Ham}
  M(z,\psi)=\frac{\beta^{-1}}{2}\sum_{i=1}^r\left( F_i(z,\psi)\right)^2.
\end{equation*}


\section{Finite ensemble controllability via Lie algebraic methods}
\label{sec_LAAC}

%

We  approach ensemble controllability from the viewpoint of geometric control theory,
in the spirit of  what has been done  in our previous publication \cite{AS20}.
 See also  preprint \cite{TG}
where the Lie algebraic methods  are  applied to a different class of systems in the context of deep learning.




We start with  basic definitions.
\begin{defn}[finite ensemble controllability]
System \eqref{eq_cafs} has   the property of finite ensemble controllability if for each
$N=1,2, \ldots$,   for each $T>0$  and for any two $N$-ples  $x_\alpha=( x_\alpha^1, \ldots , x_\alpha^N)$,  $x_\omega= ( x_\omega^1, \ldots , x_\omega^N) \in \mathcal{M}^{(N)}$ there exists
a control $u(t)=(u_1(t), \ldots , u_r(t))$ which steers the corresponding system \eqref{eq_N_ens}  from $x_\alpha$ to $x_\omega$ in time $T$.
\end{defn}

\begin{rem}
If  system \eqref{eq_N_ens}  can steer the point $x_\alpha$ to $x_\omega$ in time $T>0$ by means of a control $u(t), \ t \in [0,T]$, then
  it can do the same   in any time $T'>0$ by means of the control $\frac{T}{T'} u\left(\frac{T}{T'}t\right), \ t \in [0,T']$.
\end{rem}

For a smooth vector field $X \in \mbox{\rm Vect} \mathcal{M}$  consider its $N$-fold - the vector field on $\mathcal{M}^{(N)}$, defined as
$X^N(x^1, \ldots , x^N)=(X(x^1), \ldots ,X(x^N)).$    System \eqref{eq_N_ens} can be given form $\dot \gamma =X^N(\gamma), \ \gamma \in \mathcal{M}^{(N)}$.

For $X,Y \in \mbox{\rm Vect } \mathcal{M}$, and $N \geq 1$ we define the Lie bracket of the $N$-folds $X^N,Y^N$  on $\mathcal{M}^{(N)}$  "componentwise":
$[X^N,Y^N]=[X,Y]^N$  -  the $N$-fold  of the  Lie bracket  $[X,Y]$ of $X,Y$ on $\mathcal{M}$.
The same holds for the iterated Lie brackets.

We denote $\mbox{Lie}\{f_1, \ldots , f_r\}$  the Lie algebra  generated by the vector fields  $f_1, \ldots , f_r$, and $\mbox{Lie}\{f^N_1, \ldots , f^N_r\}$ the Lie algebra generated by their $N$-folds.

For the vector fields $f_1, \ldots f_r$ on $\mathcal{M}$,  their  $N$-folds $f_1^N, \ldots , f_r^N$  are called
 bracket generating  on $\mathcal{M}^{(N)}$, if the evaluations
 of the iterated Lie brackets  of $f_1^N, \ldots , f_r^N$ at each $\gamma=(x^1, \ldots , x^N)   \in \mathcal{M}^{(N)}$, span the tangent space $T_\gamma \mathcal{M}^{(N)}=\bigotimes_{j=1}^NT_{x_j}\mathcal{M}$.
Evidently  for $N>1$ the bracket generating property for $f_1^N, \ldots , f_r^N$ on $\mathcal{M}^{(N)}$ is strictly stronger, than the same  property for $f_1, \ldots , f_r$ on $\mathcal{M}$.

 Rashevsky-Chow theorem (\cite{ASach})   implies
 \begin{prop}
 \label{thm_brag}
If  $\dim \mathcal{M} >1$ and $\forall N \geq 1$ the $N$-folds  $f_1^N, \ldots , f_s^N$  are
 bracket generating  on $\mathcal{M}^{(N)}$, then   system \eqref{eq_cafs}
  has the property of  finite ensemble controllability on $\mathcal{M}$.
 \end{prop}

In \cite{AS20} we  proved that the latter property holds for each  $N$ and a generic $r$-ple $f_1, \ldots , f_r$ of vector fields.
In the present context  it is
more   convenient to check   a  stronger property, which  implies   the bracket generating property for any $N$.

Let us  introduce  the standard notation for the seminorms in the space of smooth vector fields on a manifold $\mathcal{M}$:
for a compact $K \subset \mathcal{M}$ and $r \geq 0$
 \[ \|X\|_{r,K} =\sup_{x \in K}\left(\sum_{0 \leq |\beta| \leq r }\left|D^\beta X(x)\right|\right), \ \|X\|_{r} =\sup_{x \in  \mathcal{M}}\left(\sum_{0 \leq |\beta | \leq r }\left|D^\beta X(x)\right|\right) .   \]

In the formulations of  controllability results  we
invoke  the following  assumptions for  the vector fields
 $f_1, \ldots , f_r \in \mbox{Vect}(\mathcal{M})$, which define   control system \eqref{eq_cafs}.

\begin{asmp}[boundedness]
\label{asmp_b}
  The vector fields   $f_j(x), \ j=1, \ldots , r$,  are    $C^\infty$-smooth   and  bounded on  $\mathcal{M}$
      together with their covariant   derivatives of each order.
\end{asmp}

\begin{asmp}[Lie algebra approximating property]
\label{asmp_a}
  A system of smooth  vector fields
$f_1, \ldots , f_r \in \mbox{Vect}(\mathcal{M})$
demonstrates the Lie algebra approximating  property,  if   $\exists m \geq 1$ such that
 for each $C^m$-smooth  vector field $Y \in \mbox{Vect}(\mathcal{M})$ and each compact $K \subset  \mathcal{M}$
  there holds:
   \[   \inf\{\left\|Y-X\right\|_{0,K}| \   X \in  \mbox{Lie}\{ f_1, \ldots , f_r\} \}=0.\]
\end{asmp}

We show that this property suffices to guarantee  finite ensemble controllability.

\begin{thm}[Lie algebra approximating  property and  finite ensemble controllability]
\label{dethm}
If  $\dim \mathcal{M} >1$ and the vector fields
$f_1, \ldots , f_r$ meet Assumptions \ref{asmp_b} and \ref{asmp_a},
 then $\forall N \geq 1$
system \eqref{eq_N_ens} is   controllable in the space $\mathcal{E}_N(\mathcal{M})$ of ensembles of $N$ points.
\end{thm}

\begin{proof}
Fix $N$.  Choose  an ensemble
$\gamma =(x^1, \ldots , x^N) \in   \mathcal{M}^{(N)}$.
We prove that the $N$-folds
$f_1^N, \ldots , f_r^N$  are
 bracket generating  at $\gamma$.

 Pick $m$ for which the Lie algebra approximating property holds.
 Consider the space $\mbox{Vect}^m(\mathcal{M})$  of $C^m$-smooth vector fields on $\mathcal{M}$
 and define for each $\gamma \in   \mathcal{M}^{(N)}$
  the evaluation  map $E_\gamma :\mbox{Vect}^m(\mathcal{M}) \mapsto T_\gamma \mathcal{M}^{(N)} $:
 \[E_\gamma (Y)=Y^N(\gamma)=\left(Y(x^1), \ldots , Y(x^N)\right) .   \]

 This linear map is obviously surjective and
 continuous with respect to  $C^0$-metric in   $\mbox{Vect}^m(\mathcal{M})$.
 By virtue of   Assumption \ref{asmp_a}  the image \linebreak
 $E_\gamma\left(\mbox{Lie}\{f_1, \ldots , f_r\}\right)$ is a dense linear subspace of
   $T_\gamma \mathcal{M}^{(N)}$ and hence must coincide with it.
\end{proof}

\begin{rem}
  Below we provide  formulations for specific  cases in which $\dim \mathcal{M}=1$.
\end{rem}

\section{Lie Algebra Strong Approximating Property.   Controllability in the Diffeomorphism Groups  and the Manifolds of Mappings}
\label{sec_grodiff}

In the previous Section we dealt with finite ensembles of points.
In this section we show   that if a stronger approximating property holds for the Lie algebra
 $\mbox{Lie}\{f_1, \ldots , f_r\}$, associated to control system \eqref{eq_cafs},
 then approximate controllability of system \eqref{eq_cafs}  holds in the group  $\mbox{Diff}^c_0$
  of diffeomorphisms on $\mathcal{M}$ and on  the manifolds  of smooth mappings of $\mathcal{M}$.

In our proofs we make  occasional  use of  few notations of chronological calculus   for the flows generated
by the time-dependent vector fields (\cite{AG77}).
In particular  for a vector field $X_t(x)$, which is smooth in $x$
and locally integrable in $t$ we denote by
$\chro_{t_0}^tX_sds$  the flow $P_t$,
 generated by the time-dependent differential equation $\dot x=X_t(x), \ P_{t_0}=I$.
  If $X_t$ is time independent: $X_t(x) \equiv X(x)$, then the flow
 is denoted by $P_t=e^{(t-t_0)X}$.  A  brief presentation of the chronological calculus  can be found in \cite{ASach}.

The following definition has been used  in \cite{AS20}.
  Put    for $\ell>0$ and a compact $K \subset \mathcal{M}$:
  \[\mbox{Lie}^\ell_{1,K}\{f_1, \ldots , f_r\}=\left\{X(x) \in \mbox{Lie}\{f_1, \ldots , f_r\}\left| \
  \|X\|_{1,K} < \ell \right.\right\}.\]

\begin{asmp}[Lie algebra strong approximating property]
\label{asmp_s}
  A system of smooth  vector fields
$f_1, \ldots , f_r \in \mbox{Vect}(\mathcal{M})$
possesses Lie algebra strong approximating  property, if
  $\exists m \geq 1$, such that for each $C^m$-smooth vector field $Y \in \mbox{Vect}(\mathcal{M})$ and each compact
  $K \subset \mathcal{M} \
\exists \ell >0$ for which:
    \begin{equation}\label{eq_lie_strong}
     \inf \left\{\left. \sup_{x \in K }\left|Y(x)-X(x) \right| \  \right|
\ X \in  \mbox{Lie}^\ell_{1,K}\left\{f_1, \ldots , f_r\right\}\right\}=0.
    \end{equation}
\end{asmp}

Denote by $\mbox{Diff}^c_0$ the connected component of the identity of
the group of the compactly supported diffeomorphisms of $\mathcal{M}$.

\begin{thm}[$C^0$-approximate controllability in the group of diffeomorphisms]
\label{thm_grodif}
   Let     $\hat P \in \mbox{Diff}^c_0(\mathcal{M})$.
  Let   $C^\infty$-smooth vector fields   $f_j(x), \ j=1, \ldots , r,$
  meet Assumptions \ref{asmp_b} and \ref{asmp_s}. %
    Then for each $K \subset \mathcal{M}$ and each  $\eps >0$ there
   exists a  control       $u(t)=(u_1(t), \ldots , u_r(t)),  \ t \in [0,T]$, such that for  the  corresponding flow
  \begin{equation}\label{eq_PTuv}
  P_t= \chro_0^t \left(\sum_{j=1}^r  f_j(x)u_j(\tau)\right)d\tau , \ x \in \mathcal{M}
  \end{equation}
      generated by system \eqref{eq_cafs},   the diffeomorphism  $P_T$    $\eps$-approximates $\hat P$ in $C^0$ on $K$:
   $\left\|\hat P-P_T\right\|_{0,K} < \eps .    $
\end{thm}

{\it Proof.}  Join the identity $I$ with $\hat P$
 by  a
 curve $t \mapsto \hat P_t(x), \ t \in [0,T]$  in  $\mbox{Diff}^c_0(\mathcal{M})$. Without loss of generality we may assume that
 $(t,x) \mapsto\hat P_t(x)$  is  $C^1$-smooth.
    The curve $t \mapsto \hat P_t(x)$ can be represented
   as a flow $\hat P_t=\chro_0^t Y_\tau d\tau$, generated by a non autonomous vector field $Y_t$,
   which is continuous in $t$; one can  take $Y_t(x)=(P_t)^{-1}_* \frac{dP_t}{dt}(x)$.

Denote by $K_t, \ t \in [0,T]$  the images of a compact set $K$
under the flow  $\hat P_t$.
As far as   for each  $ t \in [0,T]$
 condition \eqref{eq_lie_strong}   holds   for  the vector fields  $Y_t$   and   control system \eqref{eq_cafs},  then
 one can apply  Theorem 4.3 of \cite{AS20}
  to the vector field $Y_t$,  the diffeotopy $K_t, \ t \in [0,T]$ and    system \eqref{eq_cafs}.
 According to this  Theorem  for each $\varepsilon >0$  there exists a   control   $u(t)=(u_1(t), \ldots , u_r(t)),  \ t \in [0,T]$    such that for the flow \eqref{eq_PTuv}
      \[  \sup_{x \in K}\left\|\hat P(x) -P_T(x) \right\|_{0,K}  < \eps . \ \hspace{4cm } \qed  \]

The  approximation result,
we have just proved,
 can be extended from diffeomorphisms of $\mathcal{M}$ to a broader class of
 continuous maps    $\varphi:\mathcal{M} \to \mathcal{C}$.


One of  possible constructions can be realized  on the  manifold $\mathcal{M} \times \mathcal{C}$.  Consider
 the projection $p: \mathcal{M} \times \mathcal{C} \to \mathcal{C}$ and a diffeomorphic  immersion  $\imath : \mathcal{M} \to \mathcal{M} \times \mathcal{C} $. We opt for  $\imath (x)=(x,\nu), \ \forall x \in \mathcal{M}$, where $\nu$ is a selected point of $\mathcal{C} $. Let the metric $d$ on $\mathcal{M} \times \mathcal{C}$ be defined by $d=d_{\mathcal{M}}+d_{\mathcal{C}}$.

Let  $\varphi:\mathcal{M} \to \mathcal{C}$ be a continuous mapping which
is  approximately $C^1$-smoothly homotopic to  the constant mapping  $\varphi_0(x)=\nu$. This
means that in any $C^0$-neighborhood of  $\varphi$  there are $C^1$-smooth functions $\hat \varphi$, which
 are contractible to the constant function by
$C^1$-smooth homotopies $\hat \varphi_t(x), \ t \in [0,1]$:
$$ \hat \varphi_0(x) \equiv \nu, \ \hat \varphi_1(x)=\hat \varphi (x).$$

Without loss of generality we  can  limit ourselves  to the case in which $\varphi=\hat \varphi$  is $C^1$-smooth
and $C^1$-smoothly homotopic to the constant function.
Consider the  graphs of the mappings $\varphi_t(x): \ \Gamma_t = \{(x,\varphi_t(x)),  \ x \in \mathcal{M}\} \subset \mathcal{M}\times \mathcal{C}$.
For each $t$ the sets $\Gamma_t$ are  diffeomorphic to $\Gamma_0$ and to $\mathcal{M}$.
The flow
$\hat P_t$, generated on the manifold $\mathcal{M}\times \mathcal{C}$  by the vector field
$\frac{\partial \varphi_t(x)}{\partial t}\frac{\partial}{\partial c}$,
defines the diffeotopy of the graphs: $$\Gamma_t=\hat P_t(\Gamma_0), \ t \in [0,1]; \ P_1(x,\nu)=(x, \varphi(x)), \ \forall x \in
\mathcal{M}.$$


Let  control system \eqref{eq_cafs},  defined  now on   $\mathcal{M} \times \mathcal{C}$,
possess the Lie algebra strong approximating property.
By the previous theorem for each  compact $K \subset \mathcal{M}$ and each  $\varepsilon >0$ there exists a control $u(\cdot)=(u_1(\cdot), \ldots , u_r(\cdot))$,
such that for the flow
\begin{equation}\label{eq_MC}
  P_t= \chro_0^t \left(\sum_{j=1}^r  f_j(x)u_j(\tau)\right)d\tau, \ x \in \mathcal{M} \times \mathcal{C}
\end{equation}
  there holds
$\|P_1- \hat P_1\|_{0,K \times \{\nu\}}<\varepsilon$.
Then
\[\forall x \in K: \ \varepsilon >d_\mathcal{M}(p \circ P_1(x,\nu), p \circ \hat P_1(x,\nu))=d_\mathcal{M}(p \circ P_1(x,\nu),\varphi(x))  \]
and we conclude with the corollary.

 \begin{cor}
 \label{thm_maps}
Let  control system \eqref{eq_cafs},    defined on   $\mathcal{M} \times \mathcal{C}$,
meet  Assumptions \ref{asmp_b} and \ref{asmp_s}.
%
%
Then the system  is $C^0$-approximately controllable on the manifold of mappings:
for each continuous
mapping  $\varphi:\mathcal{M} \to \mathcal{C}$,  which is approxi\-mately smoothly homotopic to a constant,
each $\varepsilon >0$  and each compact $K \subset \mathcal{M}$
there    exists        $u(t)= (u_1(t), \ldots , u_r(t)),  \ t \in [0,T]$, such that for  the  corresponding flow
\eqref{eq_MC} on $\mathcal{M} \times \mathcal{C}$  there holds
   $\left\|\varphi(x) -p \circ  P_T \circ \imath (x) \right\|_{0,K} < \eps .    $
 \end{cor}

\section{Ensemble  controllable systems on Euclidean spaces $\mathbb{R}^d$,  tori $\mathbb{T}^d$ and the  $2$-dimensional sphere $\mathbb{S}$}
\label{sec_fin_ens}

In this Section we consider several manifolds, such as
 Euclidean spaces $\mathbb{R}^d$, $d$-dimensional tori $\mathbb{T}^d$ and $2$-dimensional sphere $\mathbb{S}$.
We provide examples of  control systems on the manifolds,
which    possess
controllability properties for finite ensembles and properties of approximate controllability
in the group of  diffeomorphisms of the manifolds.

{\it For the sake of brevity along the  Section we will call system ensemble controllable if  the conclusions of
Theorems \ref{dethm} and  \ref{thm_grodif} hold  for it.}

The key point of the proofs  is the verification of the Lie algebra strong approximating condition. Such a verification regards two moments.  First  we have to establish kind of "Lie rank condition" - the approximability of the vector fields by the vector fields from $\mbox{Lie}\{f_1, \ldots , f_r\}$.    The second issue is the
 regularity of  these  approximations, including boundedness of the derivatives of the approximants.



\subsection{Ensemble controllable system  in ${\mathbb  R}^d$}

Consider control-linear  system in $\mathbb{R}^d$:
\begin{equation}\label{eq_univ_ens}
  \dot{z}=\sum_{i=1}^{d}f_i(z)u_i + \sum_{i=1}^{d}g_i(z)v_i, \ z \in \mathbb{R}^d ,
\end{equation}
where
\begin{equation}\label{eq_vecfields}
  f_i(z)=e^{-\gamma(z)}\frac{\partial}{\partial z_i}, \ g_i=\frac{\partial}{\partial z_i}, \ i=1, \ldots , d,
\end{equation}
and
 \[ \gamma(z)=\frac{\langle z, z \rangle}{2}=\frac{z_1^2+\cdots +z_n^2}{2}.\]
Putting  $z=(z_1, \ldots , z_d)$,  $u=(u_1, \ldots , u_d)$,   $v=(v_1, \ldots , v_d) $
we represent equations  \eqref{eq_univ_ens}-\eqref{eq_vecfields}  in a vectorial form
 \begin{equation}\label{eq_consys}
 \dot{z}=e^{-\gamma(z)}u+v, \ z,u,v \in \mathbb{R}^d.
 \end{equation}
We call it GH-system as far as Gaussian  density function  $e^{-\gamma(z)}$ and Hermite polynomials  play important role in its study.
%

 We consider
 the       action of  system
\eqref{eq_consys} onto an    ensemble of  points \linebreak  $(x^{1}, \ldots , x^N) \in \left(\mathbb{R}^d\right)^N$.
To establish  the property of ensemble controllability we verify the Lie algebra strong approximation condition for GH system.
\begin{prop}
\label{thm_sacGH}
Vector fields \eqref{eq_vecfields} meet Assumptions  \ref{asmp_b} and \ref{asmp_s}.
\end{prop}

   \begin{proof}

  Direct computation of the iterated Lie brackets of  vector fields \eqref{eq_vecfields} gives
    \[\ad^{m_{j1}}_{g_1} \cdots \ad^{m_{jd}}_{g_d} f_j(z) = \frac{\partial^{m_j}e^{-\gamma(z)}}{\partial z_1^{m_{j1}}\ldots \partial z_d^{m_{jd}}} \frac{\partial}{\partial z_j}, \ m_j=m_{j1}+ \cdots + m_{jd} .     \]
%
 As one  knows
    \begin{equation}\label{eq_hermite}
      \frac{\partial^{m_j}e^{-\gamma(z)}}{\partial z_1^{m_{j1}}\ldots \partial z_d^{m_{jd}}}=(-1)^{m_j}H_{m_{j1}, \ldots , m_{jd}}(z)e^{-\gamma(z)}, \ z=(z_1, \ldots , z_d),
    \end{equation}
     %
     where $H_{m_{j1}, \ldots , m_{jd}}(z_1, \ldots , z_d)$ are multivariate Hermite polynomials.
     Thus  for each $j=1, \ldots , d$ and each  Hermite polynomial $H_{m_{j1}, \ldots , m_{jd}}(z)$ the vector field
     $H_{m_{j1}, \ldots , m_{jd}}(z)e^{-\gamma(z)}\frac{\partial}{\partial z_j}$ belongs to the Lie algebra
     generated by  vector fields \eqref{eq_vecfields}.

Hermite  polynomials  $\{H_{m_1,\ldots ,m_d}(z_1, \ldots , z_d)| \ m_1 \geq 0 ,  \ldots ,  m_d \geq 0\}$
 form  a complete orthogonal system in $L_2(\mathbb{R}^d)$
 with respect to the weighted scalar product
    \[ \langle f , g \rangle= \frac{1}{(2\pi)^{d/2}} \int_{\mathbb{R}^d}f(z)g(z)e^{-\gamma (z)}dx .   \]
Any function from $L_2(\mathbb{R}^d)$ can be expanded into a $L_2$-convergent series in   Hermite polynomials.
To verify the  Lie algebra strong approximating condition one has
to prove that for each sufficiently smooth
vector field  $Y(X)=\sum_{j=1}^{d}Y_j(z)\frac{\partial}{\partial z_j}$   with compact support in ${\mathbb R}^d$,
there exists   $\ell >0$ such that  for each $j=1, \ldots ,d$ and each $\eps >0$  one can find
 a linear combination $X_j$ of the functions     \eqref{eq_hermite} for which
    \[ \|X_j\|_{1,K} \leq \ell , \ \|X_j - Y_j\|_{0,K} \leq \eps .\]

 Suppose $Y(x)$ to be
$C^{[\frac{d}{2}]+2}$-smooth. Pick  its   component $Y_j(x)$  and consider
 the  orthogonal expansion of the function $Y_j(z)e^{\gamma (z)}$ in Hermite polynomials:
     \begin{equation}\label{eq_herm_ser}
      Y_j(z) e^{\gamma (z)} \sim  \sum_m c_m H_m(z), \ m=(m_1, \ldots , m_d) \in \mathbb{N}^d.
     \end{equation}
   For  $|m|=m_1+ \cdots +m_d$
   let  $S_n(z)=\sum_{m: \ |m| \leq  n} c_m  H_m$ be a partial sum of this expansion.
   Lie algebra strong approximating condition   is implied by  the following Lemma.
    \begin{lem}
    \label{thm_Herm_conv}
          For $Y_j(z)$ being $C^{\left[\frac{d}{2}\right]+2}$-smooth the functions $S_n(z)e^{-\gamma(z)}$ converge uniformly to $Y_j(z)$, as $n \to \infty$,      while  $\frac{\partial}{\partial z_i}\left(S_n(z)e^{-\gamma(z)}\right)$ converge uniformly to $\frac{\partial Y_j(z)}{\partial z_i}$  and hence are bounded by a constant $\ell$ independent of $n$.
    \end{lem}
Proof of the lemma can be found in the Appendix.
%
\end{proof}
By virtue  of  Theorems \ref{dethm} and  \ref{thm_grodif}  and Proposition \ref{thm_sacGH} there holds

\begin{thm}[ensemble controllability of GH system]
\label{thm_GH}

i) For $d>1$ system   \eqref{eq_consys} is ensemble controllable on $\mathcal{M}=\mathbb{R}^d$;

ii) For $\mathcal{M}=\mathbb{R}$  system   \eqref{eq_consys} is approximately controllable in the group of diffeomorphisms
 $\mbox{Diff}^c_0(\mathbb{R})$;

  iii)  For $\mathcal{M}=\mathbb{R}$  system   \eqref{eq_consys} can transform a finite ensemble $\left(x^1_\alpha , \cdots , x^N_\alpha\right)$ into
 another ensemble $\left(x^1_\omega , \cdots , x^N_\omega\right)$  if and only if they are equally ordered: $x^i_\alpha < x^j_\alpha \Leftrightarrow x^i_\omega < x^j_\omega ,\ \forall i,j$.
%
%
\end{thm}

\subsection{Ensemble controllability on the tori $\mathbb{T}^d$}

We start with $d=1$.
Consider   the control-linear  system on  ${\mathbb T}^1$:
\begin{equation}\label{eq_csT1}
 \dot{\varphi}=u_0+u_1 \sin \varphi + u_2 \sin 2\varphi  ,
\end{equation}
generated by
the  vector fields
\begin{equation}\label{eq_t1vecfields}
 f_0(\varphi)= \frac{ \partial}{ \partial \varphi}, \  f_1(\varphi)=\sin \varphi \frac{ \partial}{ \partial \varphi}, \ f_2(\varphi)=\sin 2\varphi \frac{ \partial}{ \partial \varphi}.
\end{equation}
 Here $\varphi$ is the angle coordinate on ${\mathbb T}^1$.

The action of  system  \eqref{eq_csT1} on  an ensemble of $N$  points $(\varphi^1_\alpha, \ldots , \varphi^N_\alpha)$
on ${\mathbb T}^1$ is  defined  by the equations
\begin{eqnarray*}
  \dot{\varphi}^j=u_0(t)+u_1(t)  \sin \varphi^j + u_2(t)  \sin 2\varphi^j, \  j=1, \ldots , N, \\
  \varphi^j(0)=\varphi^j_\alpha
  \end{eqnarray*}

\begin{lem}
\label{thm_sact1}
Vector fields \eqref{eq_t1vecfields}
 meet Assumptions  \ref{asmp_b} and \ref{asmp_s}.
 \end{lem}

\begin{proof}
Boundedness  is obvious.

We prove that
  the Lie algebra $\mbox{Lie}\{f_0,f_1,f_2\} $, generated by  vector fields  \eqref{eq_t1vecfields},  contains the vector fields
  $\sin k\varphi \frac{\partial}{\partial \varphi}, \ \cos k\varphi \frac{\partial}{\partial \varphi}, k=1, 2 , \ldots $.

  As far as $\left[ \frac{\partial}{\partial \varphi},  \sin k\varphi \frac{\partial}{\partial \varphi} \right]=k \cos k\varphi \frac{\partial}{\partial \varphi}$,  it suffices to prove  that $\sin k\varphi \frac{\partial}{\partial \varphi}$,   $k \geq 1$ are contained in $\mbox{Lie}\{f_0,f_1,f_2\} $.
  This can be  done  by induction in $k$,  given that for $k>1$
  \[\left[ \sin \varphi \frac{\partial}{\partial \varphi},  \sin k\varphi \frac{\partial}{\partial \varphi} \right]=(k-1) \sin ((k+1)\varphi)- (k+1)  \sin ((k-1)\varphi) .      \]

Consider a
 vector field $Y(\varphi)\frac{\partial}{\partial \varphi}$ on $\mathbb{T}^1$ together with  its  Fourier  expansion
 \[Y(\varphi)\frac{\partial}{\partial \varphi} \sim \frac{a_0}{2}\frac{\partial}{\partial \varphi} + \sum_{k=1}^{\infty}\left(a_k \cos k\varphi \frac{\partial}{\partial \varphi}+b_k \sin k\varphi \frac{\partial}{\partial \varphi}\right).
     \]
By the aforesaid partial sums of the series belong to
$\mbox{Lie}\{f_0,f_1,f_2\} $.
For $Y(\varphi)$ being  $C^2$-smooth  the partial sums $S_n(\varphi)$ of the Fourier series converge uniformly to $Y(\varphi)$,  as $n \to \infty$.  The derivatives $S'_n(\varphi)$ converge uniformly to  $Y'(\varphi)$ and hence are equibounded, wherefrom the Lie algebra strong approximating condition follows.
\end{proof}

To extend the construction to the $d$-dimensional torus
 $\mathbb{T}^d=\mathbb{T}^1 \times \cdots \times \mathbb{T}^1$
we introduce   the coordinates   $\varphi_1 , \ldots , \varphi_d$  in $\mathbb{T}^d$ and define the vector fields
\begin{eqnarray}\label{eq_Tdvecfields}
  f^0_i=\frac{\partial}{\partial \varphi_i}, \    f^1_i=\sin \varphi_i \frac{\partial}{\partial \varphi_i}, \  f^2_i=
  \sin 2\varphi_i \frac{\partial}{\partial \varphi_i},   \ i=1, \ldots , d;  \\
  g_i=\left( \sum_{j=1}^d \sin \varphi_j  \right) \frac{\partial}{\partial \varphi_i}, \     i=1, \ldots , d.   \nonumber
\end{eqnarray}
Consider the  control-linear  system
\begin{equation}\label{eq_ensTd}
  \dot{\varphi}_k=u_{0k}+\sin \varphi_k u_{1k}+\sin 2\varphi_k u_{2k} +\left(\sum_{j=1}^{d} \sin \varphi_j \right)v_k, \ k=1, \ldots d.
\end{equation}

\begin{lem}
  \label{thm_sintor}
  The Lie algebra $\mathcal{L}_{\mathbb{T}^d}$
  generated by  vector fields \eqref{eq_Tdvecfields} contains all the monomial vector fields of the form
  \begin{equation}\label{eq_monomial}
    \left(\prod_{i \in \mathcal{I}}             \cos k_i \varphi_i    \prod_{i \in \mathcal{I}^c}             \sin  k_i \varphi_i
                     \right)\frac{\partial}{\partial \varphi_j}, \    j=1, \ldots , d
  \end{equation}
    where $\mathcal{I} \cup  \mathcal{I}^c =\{1, \ldots , d\}, \ \mathcal{I} \cap  \mathcal{I}^c=\emptyset$.
  \end{lem}

\begin{proof}
      By the previous lemma   the monomial vector fields $\cos k_i\varphi_i \frac{\partial}{\partial \varphi_i}$, \linebreak   $\sin k_i\varphi_i \frac{\partial}{\partial \varphi_i}, \ i=1, \ldots , d,$   belong to the Lie algebra. So do  the vector fields
  \[ [f^0_j, g_i]=  \cos  \varphi_j   \frac{\partial}{\partial \varphi_i}, \   [f^0_j,[f^0_j, g_i]]=  -\sin  \varphi_j   \frac{\partial}{\partial \varphi_i} \]
for $i \neq j$ .

  If  $\sin  k \varphi_j   \frac{\partial}{\partial \varphi_i}$ for $k \leq l$ belong to $\mathcal{L}_{\mathbb{T}^d}$, then
  \[ \left[f^1_j, \sin  l \varphi_j   \frac{\partial}{\partial \varphi_i} \right] =\frac{l}{2}\left( \sin ((l+1)\varphi_j)\frac{\partial}{\partial \varphi_i} -
   \sin((l-1)\varphi_j)\frac{\partial}{\partial \varphi_i} \right) \]
and by induction in $l$ we conclude that all the monomial vector fields
$\cos   l\varphi_j   \frac{\partial}{\partial \varphi_i}$,   $\sin   l\varphi_j   \frac{\partial}{\partial \varphi_i}, \ i,j=1, \ldots , d,$ belong to $\mathcal{L}_{\mathbb{T}^d}$.

We define the degree of a monomial vector field \eqref{eq_monomial} as the  cardinality of the set  $\{i \in \{1, \ldots , d\}| k_i \neq 0\}$
and  proceed by induction in the degree.
Each monomial vector field of degree $s+1$  is either
$M(\varphi)\cos k_\alpha \varphi_\alpha \frac{\partial}{\partial \varphi_j} $ or   $M(\varphi)\sin k_\alpha \varphi_\alpha \frac{\partial}{\partial \varphi_j}$, where $M(\varphi)$ has degree $s$ and does not depend on $\varphi_\alpha$.

In  the first case if  $M(\varphi)$ does not depend on $\varphi_j$,
  and hence  $$\left[M(\varphi)\frac{\partial}{\partial \varphi_\alpha},\sin k_\alpha\varphi_\alpha\frac{\partial}{\partial \varphi_j}\right]=
 k_\alpha M(\varphi)\cos k_\alpha \varphi_\alpha \frac{\partial}{\partial \varphi_j}.$$

 If  $M(\varphi)$ depends on $\varphi_j$, then $\alpha \neq j$ and
  one can easily find a monomial $M_1(\varphi)$ of degree $s$  such that  $\frac{\partial}{\partial \varphi_j} M_1(\varphi)=M(\varphi)$.  Then  $$\left[ \cos k_\alpha \varphi_\alpha \frac{\partial}{\partial \varphi_j}, M_1(\varphi)  \frac{\partial}{\partial \varphi_j}\right]=M(\varphi)\cos k_\alpha \varphi_\alpha \frac{\partial}{\partial \varphi_j}.$$
In this way we conclude the step of induction and the proof.
\end{proof}
The Lie algebra strong approximating property for \eqref{eq_ensTd} follows from the lemma by classical approximation results for multivariate trigonometric polynomials.

In what regards the  formulation of
criteria of finite ensemble controllability there is some peculiarity   in the case of    $\mathbb{T}^1$.
    Note that for a given orientation of  $\mathbb{T}^1$ any  ensemble of $N$ points
      on $\mathbb{T}^1$ is ordered up to cyclic permutation.
      Two ensembles are equally ordered if the sequences of their indices are the same up to a cyclic permutation.

    %
%

%

\begin{thm}
\label{thm_Td}
  Control system \eqref{eq_csT1}  and \eqref{eq_ensTd} have  the following ensemble controllability properties:

i)   for $d>1$  system \eqref{eq_ensTd} is ensemble controllable on $\mathbb{T}^d$;

ii)   for $\mathcal{M}=\mathbb{T}^1$ system \eqref{eq_csT1} is  $C^0$-approximately controllable
in  $\mbox{Diff}_0(\mathbb{T}^1)$;

iii)  two  finite ensembles on $\mathbb{T}^1$
 can be steered one into another by means of  control system \eqref{eq_csT1} in  time $T>0$, if and only if they are
 equally ordered.

\end{thm}

\subsection{Ensemble controllability on the $2$-dimensional  sphere}
We construct   examples of  control systems  on  the  $2$-dimensional sphere  ${\mathbb S} \subset \mathbb{R}^3$,
 which demonstrate the property of ensemble controllability.  Both examples are related to
   the  study in \cite{AS08}   of the controllability of the Navier-Stokes equation on ${\mathbb S}$.

We  consider a Riemannian structure on ${\mathbb S}$,
   induced by the Euclidean structure of $\mathbb{R}^3 \supset {\mathbb S}$.
 If  $f: {\mathbb S} \to \mathbb{R}$  is  the  restriction onto ${\mathbb S}$ of a smooth function $F:\mathbb{R}^3 \to \mathbb{R}$,  then the spherical gradient
 \[\nabla_{\mathbb S} f(x)=\nabla F - \langle \nabla F , x \rangle Ex\]
 is the projection of the gradient $\nabla F$ onto the tangent bundle $T\mathbb{S}$ to $\mathbb{S}$.   Here $Ex$ stands for the Euler vector field in $\mathbb{R}^3$:    $Ex=\sum_{i=1}^{3}x_i\partial_i$.

In general if   $X$ is  a smooth vector field in $\mathbb{R}^3$,  then
the  projection onto  $T{\mathbb S}$
 of the restriction of $X$ to ${\mathbb S}$ is
 \[\pr_{\mathbb S} X(x)=X(x)-\langle X(x),x \rangle E(x), \ x \in {\mathbb S}.   \]
It is  smooth vector field on ${\mathbb S}$.

 Consider  standard symplectic  structure $\sigma_x(\cdot , \cdot )$ on ${\mathbb S} \subset \mathbb{R}^3$ defined by the area form. For $x \in {\mathbb S}$,
 $\xi , \eta \in T_x {\mathbb S}$ one has $\sigma_x(\xi , \eta )=\langle x , \xi , \eta \rangle$,
  where the latter trilinear form is the
  mixed product in $\mathbb{R}^3$.

 We  introduce
  the spherical divergence   $\mbox{div}_{\mathbb S}  \pr_{\mathbb S} X(x)$
 of $\pr_{\mathbb S} X(x)$  with respect to the area form $\sigma$.
To this end we consider the interior  product of the vector field $\pr_{\mathbb S} X(x)$ with the differential $2$-form $\sigma$;
  it is the $1$-form defined by
\begin{equation}\label{eq_1form}
  \eta \to \sigma(\pr_{\mathbb S} X(x) , \eta)= \langle x \times \pr_{\mathbb S} X(x), \eta \rangle =\langle x \times X(x), \eta \rangle ,
\end{equation}
 where $\times$ stands for the cross product in $\mathbb{R}^3$.
The exterior derivative of the $1$-form is the $2$-form  $\psi(x)\sigma$,  whose   coefficient $\psi(x)$
  coincides with  the spherical divergence $\mbox{div}_{\mathbb S} \pr_{\mathbb S} X(x)$.

To compute the exterior derivative we apply Stokes theorem
to the integral of the $1$-form \eqref{eq_1form}  along a closed curve on $\mathbb{S}$  and conclude  that
it equals to the flow of the curl of the vector field $x \times X(x)$ through
the spherical area circumvented  by the curve. Hence
\begin{eqnarray*}
  \mbox{div}_{\mathbb S} \pr_{\mathbb S} X(x)=\langle \mbox{curl}\left(E(x) \times X(x)\right),E(x) \rangle=\langle (\mbox{div} X)E(x),E(x)\rangle - \\  \mbox{div}  E(x)\langle X(x),E(x)\rangle=
\mbox{div} X -3  \langle X(x),E(x)\rangle .
\end{eqnarray*}
 In particular $\mbox{div}_{\mathbb S} \pr_{\mathbb S}  X(x)=\mbox{div } X$,  if $X$ is tangent to $\mathbb{S}$.

Once we have  defined spherical divergence   $\mbox{div}_{\mathbb S}$ and spherical gradient
$\nabla_{\mathbb S}$,  then spherical Laplacian of a function $f$ on ${\mathbb S}$ is defined as:
\[\Delta_{\mathbb S} f= \mbox{div}_{\mathbb S}  \nabla_{\mathbb S} f.   \]

 Consider the homogeneous harmonic polynomials on $\mathbb{R}^3 \setminus 0$ and take  their  restrictions
  onto  ${\mathbb S}$;  those   are called {\it spherical harmonics}.
  We call them  linear, quadratic, cubic, of $n$-th degree etc.,   if they
  are  restrictions of the  homogeneous polynomials of the corresponding  degree.
 Spherical harmonics are the
   eigenfunctions
  of the spherical Laplacian.

  Restriction of any  smooth function $\varphi$ in $\mathbb{R}^3$ onto  ${\mathbb S}$
  gives rise to the Hamiltonian vector field $\overrightarrow{\varphi}$ on ${\mathbb S}$, which is  defined by the relation:
    \[(\overrightarrow{\varphi}(x),\eta)=\sigma(\nabla \varphi (x),\eta)=\langle x , \nabla \varphi(x)  , \eta \rangle =(x \times \nabla \varphi (x), \eta), \ \eta \in T_x{\mathbb S};\]
        hence
        $\overrightarrow{\varphi}(x)=x \times \nabla \varphi(x), \ x \in \mathbb{S}$.

We provide an example of   Hamiltonian control system which has the property of
 approximate controllability  in the group of the area preserving diffeomorphisms on ${\mathbb S}$.

\begin{thm}
\label{thm_SDiff}
   Given three independent linear harmonics $(l^1,x)$,  $(l^2,x)$, $( l^3,x)$,  a quadratic harmonic $q(x)$, a cubic harmonic $c(x)$ and the corresponding
  Hamiltonian vector fields
  \begin{equation}\label{eq_hvf}
    \overrightarrow{l}^1(x), \overrightarrow{l}^2(x),\overrightarrow{l}^3(x), \overrightarrow{q}(x), \overrightarrow{c}(x),
  \end{equation}
 the control system
 \begin{equation}\label{eq_s2ensemble}
   \dot{x}=\sum_{i=1}^{3}\overrightarrow{l}^i(x) u_i(t)+ \overrightarrow{q}(x)v_2(t)+\overrightarrow{c}(x)v_3(t)
 \end{equation}
 is controllable in the space of finite ensembles on $\mathbb{S} $ and approximately controllable in the group  $\mbox{SDiff}_0\left({\mathbb S}\right)$ of the area preserving diffeomorphisms of $\mathbb{S} $.
\end{thm}

\begin{proof}
 The following statement has been proved in   \cite[Theorem  10.4]{AS08}.
   \begin{prop}
  \label{thm_spheharm}
    The Lie algebra generated by the Hamiltonian
  vector fields
   \eqref{eq_hvf}
  contains all the symplectic  vector fields $\overrightarrow{h}$,  which  correspond  to
  harmonic homogeneous polynomials (spherical harmonics) $h(x)$,
  and therefore  is dense in the space of  all the divergence-free  vector fields.
  \end{prop}

Spherical harmonics form a complete system in $L_2(\mathbb{S} )$.
  To prove the Lie algebra
  strong approximating condition we  consider  the  expansions of functions on $\mathbb{S} $ in  Laplace series with respect to spherical harmonics.
  We apply  a result by  M.Ganesh, I.G.Graham \& J.Sivaloganathan  \cite[Theorem 3.5]{GGS} on the
  best approximation by Laplace series
  of smooth functions on the spheres ${\mathbb S}^m$  together with their derivatives up to some order.

  \begin{lem}
  \label{thm_ragoz}
  Let $C({\mathbb S})$ be the space of continuous functions on the sphere and $\mathcal{P}_n$ be
  the space of spherical  polynomials of degree $\leq n$.
   For each $n \geq 1$ there exist continuous linear
   operator  ${\mathcal T_n}: C({\mathbb S}) \mapsto \mathcal{P}_n$
  and for every $l  \geq 0$  a constant $b_l$ such that for all $k=0, \ldots , l; \ f \in  C^l({\mathbb S})$
  \[\left\|f - {\mathcal T}_n f       \right\|_{C^k} \leq b_l\left(\frac{1}{n}\right)^{l-k}\|f\|_{C^l}.    \]

      \end{lem}
 (This  result builds on   the previous work by D.L.Ragozin   and D.J.Newman \& H.S. Shapiro;  see references in \cite{GGS}.)

     Let $Y(x)$ be  a $C^2$-smooth divergence free (Hamiltonian) vector field on ${\mathbb S}$ and $\Upsilon$
     the corresponding  $C^3$-smooth  Hamiltonian.   By lemma \ref{thm_ragoz}
     \[ \|\Upsilon -  {\mathcal T}_n \Upsilon\|_{C^2} \leq  \frac{b_2}{n}\|\Upsilon\|_{C^3}   \]
for some constant $b_2>0$.

This implies that $T_n \Upsilon$ and its first and second derivatives $DT_n \Upsilon, \ D^2T_n \Upsilon$ converge uniformly
to $ \Upsilon, \ D \Upsilon,D^2 \Upsilon$ correspondingly as $n \to \infty$.
This means that the Hamiltonian vector fields $\overrightarrow{T_n \Upsilon}$  converge uniformly to $Y$, and their  derivatives
$D\overrightarrow{T_n \Upsilon}$  converge uniformly to $DY$ as $n \to \infty$.  Hence  the
derivatives $D\overrightarrow{T_n \Upsilon}$ are equibounded, and the vector fields $\overrightarrow{T_n \Upsilon}$ are equilipschitzian.
According to  proposition \ref{thm_spheharm}  the vector fields $\overrightarrow{T_n \Upsilon}$  belong to the Lie algebra
generated by the vector fields  \eqref{eq_hvf}
and hence the Lie algebra strong approximating condition holds  for control system \eqref{eq_s2ensemble}.
\end{proof}

We now pass to finding  an example of control system, which is approximately controllable in the group of  smooth diffeomorphisms $\mbox{Diff}_0\left({\mathbb S}\right)$  of ${\mathbb S}$.

By Helmholtz-Hodge theorem  each smooth vector field $f$ on ${\mathbb S}$ can be represented as a sum of a gradient vector field
$f^\nabla=\nabla_{\mathbb S} F$ and an area-preserving (and symplectic  in the 2D case) vector field $f^\vdash$.
One may think of  constructing  the desired  example,  by joining some gradient vector fields to  Hamiltonian vector fields  \eqref{eq_hvf}.

\begin{thm}
\label{thm_Diff}
  Let $\overrightarrow{l}^1(x), \overrightarrow{l}^2(x),\overrightarrow{l}^3(x), \overrightarrow{q}(x), \overrightarrow{c}(x),$
  be the
Hamiltonian vector fields
\eqref{eq_hvf}. Let
 $\tilde l(x)=(l,x)$, $\tilde{q}(x)$, be a linear and a quadratic spherical harmonics,
 and  $\tilde l' (x)=\nabla_{S} (l,x),  \tilde{q}'(x)=\nabla_{S}\tilde{q}(x)$  be the corresponding
 gradient vector fields.

 The control system on the $2$-dimensional sphere ${\mathbb S}$
 \begin{equation}\label{eq_s2_fullensemble}
   \dot{x}=\sum_{i=1}^{3}\overrightarrow{l}^i(x) u_i(t)+ \overrightarrow{q}(x)v_2(t)+\overrightarrow{c}(x)v_3(t)
   +l' (x)w_1(t)  +\tilde{q}'(x)w_2(t)
 \end{equation}
 is controllable in the space of finite ensembles on ${\mathbb S}$  and approximately controllable in the group  $\mbox{Diff}_0\left({\mathbb S}\right)$ of the  diffeomorphisms of  ${\mathbb S}$.
\end{thm}
\begin{proof}
Finite ensemble controllability  follows immediately from the previous theorem.
Key technical result for proving controllability in $\mbox{Diff}_0\left({\mathbb S}\right)$ is
\begin{prop}
 \label{prop_harmvf}
The Lie algebra $\mathcal{L}$,  generated by the
  vector fields
 \[ \overrightarrow{l}^1(x), \overrightarrow{l}^2(x),  \overrightarrow{l}^3(x), \overrightarrow{q}(x), \overrightarrow{c}(x), \tilde l'(x), \tilde{q}'(x),\]
  contains all the Hamiltonian vector fields $\overrightarrow{h}$ and all the gradient vector fields $\nabla_{\mathbb S} h$,
   corresponding to all the  spherical harmonics $h$ on ${\mathbb S}$.
 \end{prop}

Lie algebra strong approximating property
would follow from this fact
by virtue of  approximation results for spherical harmonics and Laplace series, which we used
above in the proof of lemma \ref{thm_ragoz}.

Let  $\mathcal{L}_{div}$
be the image of the linear space $\mathcal{L}$ under the action of the linear operator $\mbox{div}_{\mathbb S}$.

\begin{prop}
\label{L_div}
The  linear space $\mathcal{L}_{div}$   contains all the spherical harmonics  on $S$.
 \end{prop}

Assuming the result  to hold, we  accomplish the proof of  proposition \ref{prop_harmvf}.
 Let $h$ be   any
 spherical harmonic, which without loss of generality we may assume   to be  homogeneous.
 %
If  $h=\mbox{div}_{\mathbb S} f$ and $f \in \mathcal{L}$ then
$\mbox{div}_{\mathbb S} \nabla_{\mathbb S} h=\alpha h$ and hence
 the  vector field  $\overrightarrow{p}=\nabla_{\mathbb S} h - \alpha f$
 is  divergence-free and therefore symplectic polynomial  vector field    on  ${\mathbb S}$.
 Without loss of generality one may assume that $p$ is a restriction onto  ${\mathbb S}$ of a harmonic  polynomial $\hat p$.  \footnote{Any restriction of a polynomial in $\mathbb{R}^3$ onto  ${\mathbb S}$ can be represented as a restriction onto  ${\mathbb S}$ of a harmonic (nonhomogeneous) polynomial}
  All polynomial symplectic vector field, which correspond to spherical harmonics,  belong to $\mathcal{L}$ by proposition \ref{thm_spheharm}  and hence  $\nabla_{\mathbb S} h \in \mathcal{L}$.

Employing Maxwell's theorem  we can reduce proposition
\ref{L_div} to a  weaker statement.

\begin{lem}
\label{thm_maxw}
The  linear space $\mathcal{L}_{div}$   contains all the spherical harmonics if and only if, for each $k$,
 $\mathcal{L}_{div}$ contains a homogeneous spherical harmonic of degree $k$.
 \end{lem}
\begin{proof}
For each $ l \in \mathbb{R}^3$
the symplectic vector field $\overrightarrow{l}$,  defines  the rotation $e^{\overrightarrow{l}} \in SO(3)$ of $\mathbb{R}^3$ and of the sphere $\mathbb{S}$.   By direct computation the adjoint action $\mbox{Ad} \left(e^{\overrightarrow{l}}\right)$ of the rotation onto a gradient vector field $\nabla_Sf(x)$ transforms it into the gradient vector field
 $ \nabla_Sf(e^{\overrightarrow{l}}(x))$.
By Maxwell's theorem (\cite{Arn})  the group of
rotations $e^{\overrightarrow{l}}$ act transitively on the space of spherical harmonics of a given degree.

By the assumptions of the theorem the Lie algebra  $\mathcal{L}$ contains  linearly independent vector fields
$\overrightarrow{l}^1(x)$, $\overrightarrow{l}^2(x)$,  $\overrightarrow{l}^3(x)$.
If a spherical  harmonic $ h$ is  homogeneous of degree $k$ and  belongs to  $\mathcal{L}_{div}$, then by the aforesaid  $\nabla_{\mathbb S} h \in \mathcal{L}$,  and
acting   onto $\nabla_{\mathbb S} h$ by   $\mbox{Ad}\left(e^{\overrightarrow{l}}\right), \ l \in \mathbb{R}^3$  we
conclude by transitivity that the   gradients of all spherical harmonics of degree $k$ are in $\mathcal{L}$  and then the harmonics themselves  are in
$\mathcal{L}_{div}$.
\end{proof}

To prove the existence of spherical harmonics of each degree in $\mathcal{L}_{div}$
 we start with  two  technical lemmas, whose proofs can  be found in  the Appendix.


%
%


\begin{lem}
\label{thm_euler_id}
  For a   harmonic polynomial  $F$, which is  homogeneous  of degree $k$ in $\mathbb{R}^3$,
there holds
\begin{eqnarray*}
  \langle \nabla F(x) , x \rangle =k F(x), \   \
  D^2F(x) x=(k-1) \nabla F(x),    \\  %
  {[ \nabla F(x), Ex]}=
  (2-k) \nabla F(x) .
\end{eqnarray*}
\end{lem}

\begin{lem}
\label{thm_div_br}
  For $f,g$, which are the restrictions  onto $\mathbb{S}$   of the
  harmonic polynomials $F,G$,  homogeneous  of degrees $k$ and $l$ in $\mathbb{R}^3$,  there holds
  \begin{eqnarray}\label{eq_div_fg}
    \mbox{div}_{\mathbb S} [\nabla_{\mathbb S} f, \nabla_{\mathbb S} g]=\mbox{div} [\nabla_{\mathbb S} f, \nabla_{\mathbb S} g]= \nonumber \\
    (k-l)(k+l+3)\left(\langle \nabla F , \nabla G \rangle|_{\mathbb S} - kl fg\right) .
  \end{eqnarray}
  \end{lem}

\begin{cor}
\label{thm_divfg}
Let $g(x)=x_3|_{\mathbb S}$ and $f(x)=F(x_1,x_2)|_{\mathbb S}$  be the restriction onto $\mathbb{S}$ of the
  harmonic polynomial $F(x_1,x_2)$ homogeneous  of degree  $k$
  in the variables $x_1,x_2$.
  Then
  \begin{equation}\label{eq_x3f}
     \mbox{div}_{\mathbb S}[\nabla_{\mathbb S} f, \nabla_{\mathbb S} g] =-(k-1)(k+4)kx_3f(x_1,x_2)
  \end{equation}
  and the right-hand side
  is a spherical harmonic polynomial
  homogeneous of degree $k+1$.
  \end{cor}
We prove   the corollary.
Formula \eqref{eq_x3f} follows from \eqref{eq_div_fg}.
As far as
   $x_3$  and $F(x_1,x_2)$ are both
 harmonic, then
\[\Delta (x_3F(x_1,x_2))(x)=2\langle \nabla x_3, \nabla F(x_1,x_2) \rangle =0\]
 and hence $x_3F(x_1,x_2)$ is harmonic in $\mathbb{R}^3$ and the restriction
$x_3F(x_1,x_2)|_{\mathbb{S}}$ is a spherical harmonic of degree $k+1$.


Now we complete the proof of Lemma \ref{thm_maxw}. As far as  the linear harmonic vector field $\tilde{l}'=\nabla_{\mathbb S} \tilde l$, and the quadratic harmonic vector field $\tilde{q}'=\nabla_{\mathbb S} \tilde q$  belong
to $\mathcal{L}$ and the group of
rotations $e^{\overrightarrow{l}}$ act transitively on the space of spherical harmonics of given degree,
we
can obtain by the action the   gradients of all the spherical harmonics of degrees $1$ and $2$  and, in particular,
 $\nabla_{\mathbb S} x_3$ and $\nabla_{\mathbb S} f(x_1,x_2)$.

Then $ [\nabla_{\mathbb S} x_3, \nabla_{\mathbb S} f(x_1,x_2)] \in \mathcal{L}$ and by  Corollary
\ref{thm_divfg}
 \[ \mbox{div}_{\mathbb S}[\nabla_{\mathbb S} x_3, \nabla_{\mathbb S} f(x_1,x_2)] =-12x_3  f(x_1,x_2)\]
  with the right-hand side being  a spherical harmonic  of degree $3$, which  belongs  to $\mathcal{L}_{div}$.
  Then by Maxwell theorem we conclude that  the gradients of all the spherical harmonics of degree $3$ belong to   $\mathcal{L}$.
  The proof can be  completed by  induction in the degree of harmonics with  Corollary \ref{thm_divfg} applied at each  induction step.
\end{proof}

\section{Appendix: proofs of technical lemmas}

\subsection{Proof of Lemma \ref{thm_Herm_conv}}
     As far as the function $Y_j(x)e^{\gamma (x)}$  is $C^{\left[\frac{d}{2}\right]+2}$-smooth,  the partial sums of series \eqref{eq_herm_ser} converge uniformly
          to it according to   \cite[Propositions 7.1.2, 7.1.5,  Corollary 7.1.3]{DX}.
          \footnote{The convergence is determined by  the interplay of two entities: the Christoffel constant $\Lambda_n$ (or the related Lebesgue constant) and the approximation error rate $E_n$
          of the function $Y_j(x)e^{\gamma (x)}$  by means of the $n$-truncations of the Hermite series.  For the uniform convergence it suffices (\cite[Proposition 7.1.2]{DX}) that $\Lambda_n \sim n^{-d}$ as $n \to \infty$ and $|E_n| \leq n^{-\frac{d}{2}-\beta}, \ \beta >0$.  For the first fact  see \cite[Proposition 7.1.5]{DX};  for the second fact, valid for  $C^{\left[\frac{d}{2}\right]+2}$-smooth functions, see \cite[Corollary  7.1.3]{DX}.}
    %
%
%
%

Thus  for each $\eps >0$ one can find sufficiently large $n$ for which  the  partial sums   $S_n(x)=\sum_{m: \ |m| \leq n}  c_{m_1, \ldots , m_d}
 H_{m_1, \ldots , m_d}(x)$
satisfy
\[\|S_n(x) - Y_j(x)e^{\gamma (x)}\|_{0,K} < \eps , \]
and hence
\[\|S_n(x)e^{-\gamma (x)} - Y_j(x)\|_{0,K} < \eps . \]

To get a bound  for the (first) partial derivative, say in $x_1$, of the functions
$S_n(x)e^{-\gamma (x)}$ we note that
\[ \frac{\partial}{\partial x_1} \left(H_{m_1, \ldots , m_d}(x)e^{-\gamma (x)}\right)=- H_{m_1+1, \ldots , m_d}(x)e^{-\gamma (x)} ,  \]
and therefore
\[ \frac{\partial}{\partial x_1}\left( S_n(x)e^{-\gamma (x)}\right)=
-\sum_{m: \ |m| \leq n}  c_{m_1, \ldots , m_d} H_{m_1+1, \ldots , m_d}(x) e^{-\gamma (x)}.   \]

We prove that the latter series
$\sum_{m}  c_{m_1, \ldots , m_d} H_{m_1+1, \ldots , m_d}(x)$ is the Fourier-Hermite series for the function
$\frac{\partial Y_j(x)}{\partial x_1}e^{\gamma(x)} \in C^{\left[\frac{d}{2}\right]+1}$ and hence
converges uniformly to $\frac{\partial Y_j(x)}{\partial x_1}$ as $n \to \infty$.

Multivariate Hermite polynomials  are factorable into the products of univariate Hermite polynomials:
   \[ H_{m_{j1}, \ldots , m_{jd}}(x_1, \ldots , x_d)=H_{m_{j1}}(x_1)  \cdots H_{m_{jd}}(x_d)\]
and therefore we may proceed  as in the univariate case. It suffices to prove that
given $ Y_j(x)e^{\gamma(x)}  \sim   \sum_{m}  c_{m} H_{m}(x), \ x \in \mathbb{R} $ it follows
\begin{equation}\label{eq_Yprime}
  Y'_j(x)e^{\gamma(x)}  \sim   -\sum_{m}  c_{m} H_{m+1}(x), \ x \in \mathbb{R} .
\end{equation}
From  the formulae for the Fourier-Hermite coefficients it follows that
\[ c_m= \frac{\int_{\mathbb{R}}Y_j(x)H_m(x)dx}{\int_{\mathbb{R}}(H_m(x))^2 e^{-\gamma(x)}dx}.\]
Since  $H'_{m+1}(x)=(m+1)H_m(x)$  we get
\[c_m= \frac{\int_{\mathbb{R}}Y_j(x)H'_{m+1}(x)dx}{(m+1)\int_{\mathbb{R}}(H_m(x))^2 e^{-\gamma(x)}dx} .    \]
From the identities   $\int_{\mathbb{R}}(H_{m}(x))^2 e^{-\gamma(x)}dx=\sqrt{2\pi}m!, \ m=0,1,2 \ldots $
we conclude that the denominator coincides with  $\int_{\mathbb{R}}(H_{m+1}(x))^2 e^{-\gamma(x)}dx$.
Integrating the numerator by parts we bring it to the form \linebreak $-\int_{\mathbb{R}}Y'_j(x)H_{m+1}(x)dx$
and thus conclude \eqref{eq_Yprime}.

By the above cited  approximation results from \cite{DX}  the partial derivatives $\frac{\partial}{\partial x_i}\left( S_n(x)e^{-\gamma (x)}\right)$ converge uniformly to $\frac{\partial Y_j(x)}{\partial x_i}$
as $n \to \infty$ and hence are upper equibounded for all  $n$.



\subsection{Proof of Lemma \ref{thm_euler_id}}
First equality is the    well known Euler identity for  homogeneous functions.

Differentiating the identity
\[ \forall t \in \mathbb{R}, \ x,y \in \mathbb{R}^3: \ \nabla F(x+t y) \cdot (x+ty)=k F(x+ty) \]
in $t$ at $t=0$ we conclude
\[D^2F(x)y \cdot x +\nabla F (x) \cdot y =k \nabla F (x) \cdot y  \     \]
and hence $\forall y: \ D^2F(x)x \cdot y=  (k-1) \nabla F (x) \cdot y$ wherefrom the second equality   follows.

The third equality follows   from the  previous two directly.

\subsection{Proof of Lemma \ref{thm_div_br}}
By direct computation with the use of Euler identity:
\begin{eqnarray*}
   [\nabla_{\mathbb S} f , \nabla_{\mathbb S} g]= [\pr \nabla F, \pr \nabla G]= [\nabla F - \langle \nabla F, x\rangle E(x), \nabla G -
\langle \nabla G, x\rangle E(x) ]= \\
 {[\nabla F(x) - (kF(x))E(x), \nabla G(x) -
(lG(x))E(x) ]}.
\end{eqnarray*}

 By simple manipulation with  application of the identities of Lemma \ref{thm_euler_id}
 we get
 \begin{eqnarray*}
   [\nabla_{\mathbb S} f , \nabla_{\mathbb S} g]=[\nabla F, \nabla G] -(2-k)(lG(x))\nabla F(x)+(2-l)(kF(x))\nabla G(x)+ \\
  (k-l)\langle  \nabla F(x) , \nabla G(x) \rangle E(x)+    kl(l-k)(F(x)G(x))E(x) .
 \end{eqnarray*}

Recall that  for $F,G$, which are  harmonic in $\mathbb{R}^3$, their gradients $\nabla F, \nabla G$
are divergence-free, and so is $[\nabla F, \nabla G]$.

Calculating the divergence of the right-hand side and using the identities of Lemma \ref{thm_euler_id}
  we get the  result we seek.

\end{document}